\documentclass{article}
								
\usepackage{amssymb,amsfonts, amsmath, amsthm,amsbsy}
\usepackage{enumitem}

\usepackage{url,breakurl}
\usepackage{tikz}

\usepackage{hyperref}

\newcommand{\define}[1]{{\bf \boldmath{#1}}}

\theoremstyle{definition}
\newtheorem{thm}{Theorem}

\newtheorem{definition}[thm]{Definition}

\newtheorem*{definition:magA}{Magnitude homology (Definition A)}
\newtheorem*{definition:magB}{Magnitude homology (Definition B)}
\newtheorem*{definition:magBb}{Magnitude homology (Definition B')}
\newtheorem{lemma}[thm]{Lemma}
\newtheorem{corollary}[thm]{Corollary}

\newtheorem{proposition}[thm]{Proposition}
\newtheorem{example}[thm]{Example}

\newtheorem{remark}[thm]{Remark}

\newcommand{\N}{\mathbb{N}}

\newcommand{\V}{\mathcal{V}}
\newcommand{\Z}{\mathbb{Z}}
\newcommand{\op}{\mathrm{op}}
\newcommand{\ch}{\mathrm{ch}}
\newcommand{\ob}{\mathrm{ob}}
\newcommand{\F}{\mathbb{F}}
\newcommand{\K}{\mathbb{K}}

\newcommand{\simp}{\mathrm{sim}}
\newcommand{\tot}{\mathrm{sing}}
\newcommand{\UU}{\mathcal{U}}
\newcommand{\VV}{\mathcal{V}}
\newcommand{\sSet}{\mathrm{sSet}}
\newcommand{\Set}{\mathrm{Set}}
\newcommand{\sAb}{\mathrm{sAb}}
\newcommand{\Ab}{\mathrm{Ab}}
\newcommand{\Fun}{\mathrm{Fun}}

\newcommand{\id}{\mathrm{id}}

\newcommand{\Vect}{\mathrm{Vect}}

\newcommand{\colim}{\mathrm{colim}}
\renewcommand{\lim}{\mathrm{lim}}

\newcommand{\NN}{\mathbb{N}}

\begin{document}

\title{Magnitude meets persistence. Homology theories for filtered simplicial sets}
\date{}
\author{Nina Otter
\thanks{UCLA Mathematics Department, Los Angeles, CA, USA 90095-1555, {\em email:} \url{n.otter@qmul.ac.uk} }}

\maketitle


\begin{abstract}

The Euler characteristic is an invariant of a topological space that  in a precise sense captures  its canonical notion of size, akin to the cardinality of a set. The Euler characteristic is closely related to the homology of a space, as it can be expressed as the alternating sum of its Betti numbers, whenever the sum is well-defined. Thus, one says that homology categorifies the Euler characteristic. 
In his work on the generalisation of cardinality-like invariants,  Leinster introduced the magnitude of a metric space, a real number that counts the ``effective number of points'' of the space and has been shown to encode many invariants of metric spaces from integral geometry and geometric measure theory. In 2015, Hepworth and Willerton introduced a homology theory for metric graphs, called magnitude homology, which categorifies the magnitude of a finite metric graph.  This work was subsequently generalised to enriched categories  by Leinster and Shulman, and the homology theory that they introduced categorifies magnitude for arbitrary finite metric spaces. When studying a metric space, one is often only interested in the metric space up to a rescaling of the distance of the points by a non-negative real number. The magnitude function describes how  the effective number of points changes as one scales the distance, and it is completely encoded by
  magnitude homology.  When studying a finite metric space in topological data analysis using persistent homology, one approximates the space through a nested sequence of simplicial complexes  so as to  recover topological information about the space by studying the homology of this sequence. Here we relate magnitude homology and persistent homology as two different ways of computing homology of filtered simplicial sets.

\end{abstract}

\section{Introduction}
In a letter to Goldbach written in 1750, Euler \cite{Euler} noted that for any polyhedron consisting of $F$ regions, $E$ edges and $V$ vertices one obtains 
$V - E + F = 2$. This sum is known as the \define{Euler characteristic} of the polyhedron. 
While  one usually  first encounters the Euler characteristic in relation to topological spaces, one can more generally define the Euler characteristic of an object in any symmetric monoidal  category \cite{M01}, and this can be thought of as its canonical size, a ``dimensionless'' measure. The irrelevance of topology for the notion of Euler characteristic, and how it should be  thought of as an invariant giving a measure of the size or cardinality of an object was  made precise among others by Schanuel \cite{S90}.

In his work on the generalisation of the Euler characteristic as a cardinality-like invariant, Leinster \cite{L08} introduced an invariant for finite categories generalising work done by Rota on posets. The invariant introduced by Leinster generalises both the cardinality of a set, as well as the topological  Euler characteristic. 
In subsequent work \cite{L13} Leinster generalised this invariant to enriched categories, calling it \define{magnitude}.

Here we are interested in the magnitude of metric spaces.
In 1973 Lawvere \cite{L73} observed that every metric space is a category enriched over the monoidal category $\boldsymbol{[0,\infty]}^\op$ with objects non-negative real numbers, and a morphism $\epsilon'\to \epsilon$ whenever $\epsilon'\geq \epsilon$, with tensor product given by addition. Such enriched categories are called ``Lawvere metric spaces'', and a Lawvere metric space is the same thing as  an extended quasi-pseudometric space. 
 The magnitude of a metric space is a real number that can be thought of as measuring the ``effective number of points'' of the space, see \cite[Proposition 2.8]{LM17}. The \define{magnitude function} describes how  the effective number of points changes as one scales the distances of the points of the metric space by a non-negative real number.

 The Euler characteristic of a topological space $X$ is closely related to the singular homology of the space, as it can be expressed as the alternating sum of its Betti numbers
\[
\chi(X)=\sum_{i=0}^\infty (-1)^i\beta_i(X) \, ,
\]
whenever the sum and the summands are finite.
 One then says that homology \define{categorifies} the Euler characteristic.
Thus, a natural question to ask is whether there is a homology theory for metric spaces that categorifies in an analogous way the magnitude.  
Hepworth and Willerton answered this question in the affirmative for finite metric spaces associated to  graphs, by introducing  \define{magnitude homology} for graphs \cite{HW17}. Their work was subsequently extended to arbitrary metric spaces by Leinster and Shulman \cite{LS17}, who define  magnitude homology for arbitrary metric spaces as a special case of Hochschild homology for enriched categories. 
When the metric space is finite, this  homology theory
 categorifies the magnitude.

In a first version of their manuscript\footnote{In the published version of the manuscript \cite{LS17}, these questions  are discussed under item (7) in Section 8.}, Leinster and Shulman  listed a series of open problems, of which two were as follows:
{\it
\noindent
\begin{itemize}  
\item Magnitude homology only ``notices'' whether the triangle inequality is a strict equality or not. Is there a ``blurred'' version that notices ``approximate equalities''?
\item Almost everyone who encounters both magnitude homology and persistent homology feels that there should be some relationship between them. What is it?
\end{itemize}
}
\label{questions}

Here we give an answer to these questions, which we show are intertwined: we define a blurred version of magnitude homology, and show that it is the persistent homology with respect to a certain filtered simplicial set that approximates the Vietoris--Rips simplicial set (as shown in the proof of Theorem \ref{T:nerve vr}), and thus satisfies theoretical guarantees that might make it suitable for the study of data. Ordinary and blurred versions of magnitude homology are morally very different homology theories associated to filtered simplicial sets:  the ordinary version forgets the information given by the filtration of the simplicial set, which is exactly the ``persistent'' information captured by persistent homology, and hence the blurred version of magnitude homology.  Finally, we relate  blurred and ordinary magnitude homology with Vietoris homology, by taking their categorical limits, and we show that the limit of blurred magnitude homology coincides with Vietoris homology, while the limit of magnitude homology is trivial.

We note that while in Vietoris homology and persistent homology one works with simplicial complexes, the definition of  magnitude homology is based on  simplicial  sets.
Simplicial complexes present advantages from the computational point of view, as a simplex can be uniquely specified by listing its vertices, but from the theoretical point of view simplicial sets are better suited. In Section \ref{S:SS} we  explain how to a given  simplicial complex one can assign a simplicial set such that their geometric realisations are homotopy equivalent. To make this manuscript accessible to a broad audience,  we have taken special care in introducing  notions related to both magnitude homology, as well as persistent homology.

\subsection{Structure of the paper}

The paper is structured as follows:
\begin{itemize}
\item We cover preliminaries about simplicial complexes and  simplicial sets in Section \ref{S:SS}; enriched categories and Lawvere metric spaces in Section \ref{S:lawvere};  filtered simplicial sets in Section \ref{S:filtered simplicial sets};  persistent as well as graded objects in Section \ref{S:persistent graded}; and coends in Section \ref{S:coends}.
\item In Section \ref{S:MH} we give the definition of magnitude homology for metric spaces as a special case of Hochschild homology following \cite{LS17} (see Definition A in Section \ref{SS:hochschild}), and then introduce an alternative definition based on the enriched nerve (Definition B' in Section \ref{SS:alternative}), and show that they are equivalent in Proposition \ref{P:MH eq}. 

\item In Section \ref{S:PH} we give a general definition of persistent homology, while in Section \ref{S:MH and PH} we introduce  blurred magnitude homology, taking as starting point the alternative definition of magnitude homology (Definition B' in Section \ref{SS:alternative}), and show that it is the persistent homology taken with respect to the enriched nerve.

\item In  Section \ref{S:limit}. we relate blurred and ordinary magnitude homology to Vietoris homology.\end{itemize}


\section{Simplicial complexes and simplicial sets}\label{S:SS}
The number of researchers who have a working knowledge of both simplicial complexes and simplicial sets is arguably  small, therefore
here we recall the basic notions and definitions. 
 Simplicial complexes and simplicial sets can be seen as combinatorial versions of topological spaces; they are related to topological spaces by the geometric realisation. We first recall the definitions of simplicial complexes, simplicial sets and the corresponding geometric realisations. We then  discuss how one can assign a simplicial set to a simplicial complex in such a way that the corresponding geometric realisations are homeomorphic.

\begin{definition}\label{D:simplicial complex}
A \define{simplicial complex} is a tuple $K=(V,\Sigma)$ where $V$  is a set, and $\Sigma$ is a set of non-empty finite subsets of $V$ such that:
\begin{enumerate}[label=(\roman*)]
\item for all $v\in V$ we have that $\{v\}\in \Sigma$
\item $\Sigma$ is closed with respect to taking subsets.
\end{enumerate}
The elements of $\Sigma$ with cardinality $n+1$ are called  \define{$n$-simplices} of $K$. The elements of $V$ are called \define{vertices} of $K$.
Given two simplicial complexes  $K=(V,\Sigma)$ and $K'=(V',\Sigma')$, a \define{simplicial map} $K\to K'$ is a map $f\colon V\to V'$ such that for all $\sigma\in \Sigma$ we have $f(\sigma)\in \Sigma'$.
\end{definition}

\begin{remark}
We note that if one wants the $0$-simplices to coincide with the vertices of a simplicial complex, then  condition (i) in Definition \ref{D:simplicial complex}  cannot be dispensed of; while condition (ii) implies that all vertices contained in simplices are in $\Sigma$, condition (i) guarantees that these are the only vertices. Often in the topological data analysis literature one finds a definition of simplicial complex as a variant of Definition \ref{D:simplicial complex} in which condition (i) is omitted, and in such a definition  one thus allows  vertices that are not $0$-simplices. Such simplicial complexes are studied in combinatorial commutative algebra, where they are  known to correspond to square-free monomial ideals, see \cite[Chapter 1]{MS05}. One could give a definition equivalent to Definition \ref{D:simplicial complex} by only requiring closure under taking subsets as follows: let $\Sigma$ be a family of non-empty finite sets closed under taking subsets, and let $V(\Sigma)=\bigcup \Sigma$. Then $(V(\Sigma),\Sigma)$ is a simplicial complex according to Definition  \ref{D:simplicial complex}.
 \end{remark}

To define simplicial sets, we first need to introduce the ``simplex category'' $\Delta$. Consider the category with objects finite non-empty totally ordered sets, and morphisms given by order preserving maps. The skeleton of this category is denoted by $\Delta$ and called \define{simplex category}. In other words, $\Delta$ has objects given by a  totally ordered set
$[n]=\{0,1,\dots , n\}$ 
for every natural number $n$, and morphisms order-preserving maps.

\begin{definition} Denote by $\mathbf{\Set}$  the category with objects sets and morphisms maps of sets. 
A \define{simplicial set} is a  functor $S\colon \Delta^\op\to \Set$. The elements of $S(n)$ are called \define{$n$-simplices}.
\end{definition}

\noindent
Explicitly, one can show that a simplicial set is a collection of sets $\{S_n\}_{n\in \mathbb{N}}$ together with so-called \define{face} maps 
\[
d_i\colon S_n\to S_{n-1}
\]
and \define{degeneracy} maps
\[
s_i\colon S_n\to S_{n+1}
\]
for all $0\leq i \leq n$, that satisfy certain compatibility conditions, see  \cite[Def.\ 1.1]{C71}.

The geometric realisation functor gives a canonical way to associate a topological space to a simplicial complex or set. For this, one first chooses a topological model for $n$-simplices, namely the standard $n$-simplex $\Delta^n$:
\begin{definition}
The \define{standard $n$-simplex} is the subset of Euclidean space
\[
\Delta^n=\left \{ (x_0,\dots , x_n)\in \mathbb{R}^n \mid \sum_{i=0}^n x_i=1 \, , \text{ and } 0\leq x_i \leq 1 \text{ for all } i\right \}\, .
\]

\noindent

Furthermore, there are \define{face inclusions} 
\[
\sigma_i\colon \Delta^n\to \Delta^{n+1}\colon (x_0,\dots , x_n)\mapsto (x_0,\dots , x_{i-1},0,x_i,\dots, x_n) \, .
\]
\end{definition}

\noindent
Then,  to define the geometric realisation  one
proceeds to glue together standard simplices:

\begin{definition}
Given a simplicial complex $K=(V,\Sigma)$, we choose a total order on the set of vertices $V$, and we define its \define{geometric realisation $|K|$} to be  the  quotient space 
\[
\bigcup_{\sigma \in \Sigma}\Delta^{|\sigma|-1}\times \{\sigma\}/\sim
\]

\noindent
where  $\bigcup_{\sigma \in \Sigma}\Delta^{|\sigma|-1}\times \{\sigma\}$
 is endowed with the disjoint union space topology, while the equivalence relation $\sim$ is the transitive closure of the following relation 
\[
\Big\{\Big((x,d_i(\sigma)),(\sigma_i(x),\sigma)\Big) \mid  x\in \Delta^{|\sigma|-1}, \text{ and } \; \sigma\in \Sigma  \Big\}\, .
\]
In other words, whenever  $\sigma\subseteq \tau$, we use the face inclusions to identify the copy of the standard simplex corresponding to $\sigma$ with a subset of the  copy of the standard simplex corresponding to $\tau$.

Similarly, given a simplicial set  $S\colon \Delta^\op\to \Set$, its \define{geometric realisation $|S|$} is the quotient space

\[
\bigcup_{n\in \mathbb{N}}\Delta^n\times S_n/\sim
\]

\noindent
where the equivalence relation $\sim$ is the transitive closure of the union of the relations
\[
\Big\{\Big((x,d_i(\sigma)),(\sigma_i(x),\sigma)\Big) \colon  x\in \Delta^n \text{ and } \sigma \in S_{n+1}\Big\}\, , \text{ and }
\]

\[
\Big\{\Big((x,s_i(\sigma)),(\delta_i(x),\sigma)\Big) \colon  x\in \Delta^n \text{ and } \sigma \in S_{n-1}\Big\}\, . 
\]
\end{definition}

Now, given a  simplicial complex $K=(V,\Sigma)$, we assign to it a  simplicial set so that its geometric realisation is homeomorphic to that of $K$.

\begin{definition} Let $K=(V,\Sigma)$ be a simplicial complex.
Choose a  total order on $V$.
Define 

\[
K^\simp_n=\{(x_0,\dots , x_n)\mid \{x_0,\dots, x_n\}\in \Sigma \text{ and } x_0\leq \dots \leq x_n\} \, ,
\]
\noindent
 and for  $0\leq i \leq n$ let 
\[
d_i\colon K^{\simp}_{n}\to K^\simp_{n-1}\colon (x_0,\dots , x_n)\mapsto (x_0,\dots , \hat{x_{i}},\dots , x_n) \, ,
\]
where $ \hat{x_{i}}$ means that that entry is missing, and let
\[
s_i\colon K^\simp_{n}\to K^\simp_{n+1}\colon (x_0,\dots , x_{n})\mapsto (x_0,\dots ,x_i,x_i,\dots , x_{n}) \, .
\]

\end{definition}
\noindent
It is then easy to   show that $\{K^\simp_n\}_{n\in \mathbb{N}}$ together with the maps $d_i$ and $s_i$ is a simplicial set. We denote this simplicial set by $K^\simp$. Furthermore, we have:

\begin{lemma}\label{L:geom real simpl}
The geometric realisations of $K^\simp$ and $K$ are homeomorphic.
\end{lemma}

\begin{proof}
This is easy to see, since the non-degenerate simplices are in bijection, and all degenerate simplices are in the image of some non-degenerate simplex. For more details, we refer the reader to  \cite{C71}.
\end{proof}

The assignment $K\mapsto K^\simp$ is not functorial, since it depends on  the choice of a total order on $V$. 
One can assign a simplicial set to a simplicial complex  in a functorial way, so that their geometric realisations are homotopy equivalent rather than homeomorphic, however this is at the cost of adding many more simplices. Here we discuss one such functorial assignment, which will play a crucial role in relating a homology theory introduced by Vietoris with blurred magnitude homology, in Section \ref{S:limit}.

\begin{definition}
Let $K=(V,\Sigma)$ be a simplicial complex.
Define 

\[
K^\tot_n=\left\{(x_0,\dots , x_n)\mid \{x_0,\dots, x_n\}\in \Sigma\right \} \, ,
\]
\noindent
 and for  $0\leq i \leq n$ let 
\[
d_i\colon K^{\tot}_{n}\to K^\tot_{n-1}\colon (x_0,\dots , x_n)\mapsto (x_0,\dots , \hat{x_{i}},\dots , x_n) \, ,
\]
where $ \hat{x_{i}}$ means that that entry is missing, and let
\[
s_i\colon K^\tot_{n}\to K^\tot_{n+1}\colon (x_0,\dots , x_{n})\mapsto (x_0,\dots ,x_i,x_i,\dots , x_{n}) \, .
\]

\end{definition}

One has that $\{K^\tot_n\}_{n\in \mathbb{N}}$ together with the maps $d_i$ and $s_i$ is a simplicial set, which we denote by $K^\tot$. Furthermore, we have:

\begin{proposition}\label{L:geom real tot simpl}
The geometric realisations of $K^\tot$, $K^\simp$ and $K$ are pairwise homotopy equivalent.
\end{proposition}

\begin{proof}
The simplicial set $K^\tot$ can be thought of as the analogon of the singular simplicial set associated to a topological space. Thus, it should not be too surprising that the geometric realisation of $K^\tot$ is homotopy equivalent to that of $K$. While this fact is well-known in the algebraic topology community, we were unable to find a reference. Two different proofs of this fact are provided in unpublished notes by Camarena \cite{camarena}. The remaining part of the claim follows from Lemma \ref{L:geom real simpl}.
\end{proof}


\section{Enriched categories and Lawvere metric spaces}\label{S:lawvere}
An ordinary (small) category $C$ is given by a set of objects, and for every pair of objects $x,y$ a set of morphisms $C(x,y)$, together with composition maps
\[
C(x,y)\times C(y,z)\to C(x,z)
\]
and maps assigning to every object $x$ its identity morphism
\[
\{\star\}\to C(x,x) \, ,
\]
 such that the composition of morphisms is associative and the identity morphism for every object is the neutral element for this composition. Let $\boldsymbol{\V}$ be a monoidal category  with tensor product $\otimes_{\boldsymbol{\V}}$ and unit $1_{\boldsymbol{\V}}$. A (small) \define{category enriched over $\boldsymbol{\V}$}  (or \define{$\boldsymbol{\V}$-category}) is  a generalisation of an ordinary category:  we still have  a set of objects, but now  for every pair of objects $x,y$ we are given an object $C(x,y)$ in $\boldsymbol{\V}$, together with composition and identity assigning morphisms in $\boldsymbol{\V}$, namely
 \[
 C(x,y)\otimes_{\boldsymbol{\V}} C(y,z)\to C(x,z)
 \]
 and
 
 \[
1_{\boldsymbol{\V}}\to C(x,x) \, ,
 \]
 which satisfy associativity and unitality conditions. When $\boldsymbol{\V}$ is the category of sets, a category enriched over $\boldsymbol{\V}$ is an ordinary category. We note that while an enriched category is in general not a category, it has an ``underlying''  category, see \cite{kelly} for details.

In \cite{L73} Lawvere observed that  any metric space is an  enriched category:

\begin{definition}
Let  $\boldsymbol{[0,\infty]}^\op$
 denote the symmetric monoidal category with objects given by the extended non-negative real numbers (that is, elements of $[0,\infty]$), exactly one morphism $\epsilon'\to \epsilon$ if $\epsilon'\geq \epsilon$, tensor product given by addition, and unit by $0$. 
A \define{Lawvere metric space} is a small category enriched over $\boldsymbol{[0,\infty]}^\op$.
\end{definition}

In other words, a Lawvere metric space is given by a set $X$, together with for all $x,y\in X$ a number $X(x,y)\in [0,\infty]$, and for all  $x,y,z\in X$ a morphism
\begin{equation}\label{E:lawvere1}
X(x,y)+X(y,z)\to X(x,z)
\end{equation}
as well as a morphism
\begin{equation}\label{E:lawvere2}
0\to X(x,x) \,.
\end{equation}

Equation \eqref{E:lawvere1} is the triangle inequality, while Equation \eqref{E:lawvere2} implies that $X(x,x)=0$. Thus, a Lawvere metric space is the same thing as an extended (since we are allowing infinite distances) quasi-pseudometric space (as distances are not necessarily symmetric, and we allow distinct elements to have zero distance).



\section{Filtered simplicial sets}\label{S:filtered simplicial sets}
Given a metric space $(X,d)$ we are interested in  associating to it filtered simplicial sets, namely  functors
$S(X)\colon {\boldsymbol{[0,\infty]} }\to \sSet$.
Two main examples  that we  consider in this paper are the enriched nerve and the Vietoris--Rips simplicial set.  We next recall their definitions.

\begin{definition}\label{D:enriched nerve metric}
Let $(X,d)$ be a metric space. The \define{enriched nerve} of $X$ is  the functor $N(X)\colon  {\boldsymbol{[0,\infty]} }\longrightarrow \sSet$ such that for any $\epsilon\in [0,\infty]$ the simplicial set $N(X)(\epsilon)$  has  set of $n$-simplices given by 
\[
N(X)(\epsilon)_n=\left\{(x_0,\dots , x_n)\mid x_i\in X \text{, and } \sum_{i=0}^{n-1} d(x_i,x_{i+1})\leq \epsilon)\right\} \, 
\]
and the obvious degeneracy and face maps. Further, for any $\epsilon \leq \epsilon'$ the simplicial maps  
$N(X)(\epsilon \leq \epsilon')\colon N(X)(\epsilon)\to N(X)(\epsilon')$ are the canonical inclusion maps.
\end{definition}

When adding up pairwise lengths of an ordered tuple, we will often talk about the ``length'' of the tuple:

\begin{definition}
Let $(X,d)$ be a metric space. The \define{length} of an ordered tuple $(x_0,\dots , x_n)$ of elements of $X$ is $\sum_{i=0}^{n-1}d(x_i,x_{i+1})$.
\end{definition}

\begin{definition}\label{D:VR}
Let $(X,d)$ be a metric space.
The \define{Vietoris--Rips simplicial set} of $X$ is the functor 
$V^{\tot}(X)\colon  {\boldsymbol{[0,\infty]} }\longrightarrow \sSet$
 with set of $n$-simplices given by
\[
V^\tot(X)(\epsilon)_n=\big \{(x_0,\dots , x_n)\mid d(x_i,x_j)\leq \epsilon \text{ for all } i,j\in \{0,\dots , n\}\big \} \, 
\]
and the obvious degeneracy and face maps. Furthermore, for  any $\epsilon \leq \epsilon'$ the simplicial maps  
$V^\tot(X)(\epsilon \leq \epsilon')\colon V^\tot(X)(\epsilon)\to V^\tot(X)(\epsilon')$ are the canonical inclusion maps.

\end{definition}

\begin{remark}\label{R:why monoidal poset}
We note that we are indeed interested in studying simplicial sets filtered by the monoidal category ${\boldsymbol{[0,\infty]} }$, and not merely by the category associated to the poset  $([0,\infty], \leq)$. Firstly, the enriched nerve is the generalisation of the nerve of a category to the enriched setting, and it can be  defined,  using the Yoneda embedding, as a simplicial object in the category of presheaves $\Set^{\boldsymbol{[0,\infty]}}$, see  Section \ref{SS:enriched nerve}.
 Secondly, as we will explain in the next section, a fundamental observation in  persistent homology is that  functors $\boldsymbol{[0,\infty]} \to \K\Vect$ can be identified with graded modules over a certain monoid ring, and implicit in this identification is the fact that the poset has a monoid structure compatible with the order. The monoidal structure is also crucial for the study of questions related to stability in persistent homology, see \cite{BSS15}.
\end{remark}

\subsection{The  nerve of an enriched category}\label{SS:enriched nerve}

 We recall  the construction of the nerve for enriched categories, and, in particular, for metric spaces.
 The author learned about this construction from John Baez, and the following discussion is due to him.

Given an ordinary category $C$, the \define{nerve} $N(C)$ is a simplicial set
whose $n$-simplices are composable $n$-tuples of morphisms in $C$:
\[      x_0 \stackrel{f_1}{\longrightarrow} x_1 \stackrel{f_2}{\longrightarrow} \; \cdots \;
\stackrel{f_{n-1}}{\longrightarrow} x_{n-1}
\stackrel{f_n}{\longrightarrow} x_n .\]
In other words, the set of $n$-simplices of the nerve is a disjoint union of products:
\begin{equation}
\label{nerve}    N(C)_n = \bigsqcup_{x_0, \dots, x_n \in \ob C} C(x_0, x_1) \times \cdots 
\times C(x_{n-1},x_n) .
\end{equation}
The face maps in $N(C)$ are defined using composition, while the degeneracy maps are
defined using identity morphisms.  

To generalise this concept to categories enriched over an arbitrary monoidal category $\boldsymbol{\V}$ one can  proceed as follows.
The product of sets in Equation \eqref{nerve} should be replaced by the tensor product $\otimes_{\boldsymbol{\V}}$ in $\boldsymbol{\V}$.  
The disjoint union of sets is a special case of a coproduct.  While $\boldsymbol{\V}$ may not have coproducts, the category of presheaves on $\boldsymbol{\V}$, denoted $\widehat{\boldsymbol{\V}}$ does. The objects of this category are functors $F \colon \boldsymbol{\V}^\op \to \Set$, called \define{presheaves} on $\boldsymbol{\V}$.   The morphisms are natural transformations.   

The category $\widehat{\boldsymbol{\V}}$ contains $\boldsymbol{\V}$ as a subcategory via the Yoneda embedding
\[    Y \colon \boldsymbol{\V} \to \widehat{\boldsymbol{\V}} \]
which sends each object $\epsilon \in \ob\boldsymbol{\V}$ to the so-called representable presheaf
\[          \boldsymbol{\V}(-,\epsilon) \colon \boldsymbol{\V}^\op \to \Set. \]
The coproducts in $\widehat{\boldsymbol{\V}}$ are computed objectwise: if $\left\{F_j\right\}_{j \in J}$ is a collection of presheaves on $\boldsymbol{\V}$, their coproduct
is given by
\[        \left(\bigsqcup_{j \in J} F_j\right)(\epsilon) = 
\bigsqcup_{j \in J} F_j(\epsilon)  \]
for all $\epsilon\in \ob\boldsymbol{\V}$, see  \cite[Sec.\ V.3]{lane1998categories} for more details. 
Now we can generalise the nerve 
to a category enriched over
$\boldsymbol{\V}$:

\begin{definition} 
Let $C$ be a $\boldsymbol{\V}$-category. The \define{enriched nerve of $C$} is the functor 
\[    N(C) \colon \boldsymbol{\V}^\op \to \sSet  \]
where for each $\epsilon\in \ob \boldsymbol{\V}$ the set of $n$-simplices
is given by
\begin{equation}\notag
\label{enriched_nerve}    
 N(C)(\epsilon)_n = \bigsqcup_{x_0, \dots, x_n \in \ob C} \boldsymbol{\V}(\epsilon, C(x_0, x_1) \otimes_{\boldsymbol{\V}} \cdots 
\otimes_{\boldsymbol{\V}}  C(x_{n-1},x_n) )  \, .
\end{equation}
  The  maps
\[   d_i \colon N(C)(\epsilon)_n \to N(C)(\epsilon)_{n-1}  \qquad i = 0, \dots , n \]
are defined using composition morphisms in $C$, while the degeneracy maps 
\[     s_i \colon N(C)(\epsilon)_n \to N(C)(\epsilon)_{n+1}  \qquad i = 0 , \dots , n \]
are defined using identity-assigning morphisms, all in a manner closely mimicking the usual
nerve. 
\end{definition}

When $\boldsymbol{\V} = \boldsymbol{[0,\infty]}^\op$ and $X$ is a  $\boldsymbol{\V}$-category, the set
\[   \boldsymbol{\V}\big(\epsilon, X(x_0, x_1) \otimes_{\boldsymbol{\V}} \cdots  \otimes_{\boldsymbol{\V}} X(x_{n-1},x_n) \big) = 
  \boldsymbol{\V}\big(\epsilon,  d(x_0, x_1) + \cdots + d(x_{n-1},x_n) \big)  \]
is a singleton if
\[
    \epsilon\ge d(x_0, x_1) + \cdots + d(x_{n-1},x_n)  \]
and empty otherwise.  Thus, we have a canonical isomorphism

\begin{equation}\label{E:enr ner}
    N(X)_n(\epsilon) \cong \Big\{ (x_0, \dots, x_n) \mid \; d(x_0, x_1) + \cdots + d(x_{n-1},x_n) \le \epsilon\Big\}.
 \end{equation}

We can take the isomorphic set in \eqref{E:enr ner} as the set of $n$-simplices in the enriched nerve $N(X)$ associated to a metric space $X$, and thus obtain the enriched nerve as defined in Definition \ref{D:enriched nerve metric}.

\section{Persistent vs.\ graded objects}\label{S:persistent graded}

Our aim is 
 to study the homology of  filtered simplicial sets such as those introduced in Section \ref{S:filtered simplicial sets}, and we are thus interested in functors ${\boldsymbol{[0,\infty]} }\to \ch_\Ab$. 
 Since such functors are the central object of study in persistent homology,  we introduce the following definition:

\begin{definition}\label{D:persistent objects}
Let $C$ be a small category, and $(P,\leq, +)$ a monoidal poset, that is, a poset together with a monoid structure compatible with the order. We identify $(P,\leq, +)$  with the symmetric monoidal category $\boldsymbol{P}$ with objects given by the elements of $P$, exactly one morphism $p'\to p$  if $p'\leq p$, and tensor product given by $+$.  A  functor $ \boldsymbol{P}\to C$ is called a \define{$\boldsymbol{P}$-persistent} element of the set of objects of $C$. 
\end{definition}

\begin{example}\label{E: persistent modules}
Consider $(\N,\leq, +)$ where  $\leq$ and $+$ are the usual order and addition on the natural numbers. Further, let $C=\K \Vect$ be the category of vector spaces over a field $\K$ together with $\K$-linear maps. There is an isomorphism of categories between the functor category of $\NN$-persistent vector spaces over $\K$ and the category of $\N$-graded modules over the polynomial ring $\K[x]$. 
Similarly, when we consider the monoidal poset  $([0,\infty],\leq,+)$ where  $\leq$ and $+$ are the usual order and addition on real numbers, there is an isomorphism of categories between the functor category of $\boldsymbol{[0,\infty]}$-persistent vector spaces and the category of  modules graded by $([0,\infty],\leq +)$ over the monoid ring $\K[([0,\infty],+)]$. Furthermore, finitely presented modules correspond to persistent vector spaces of ``finitely presented type'' \cite{CK17}. This is known as the Correspondence Theorem in the persistent homology literature, and $\mathbb{N}$-, as well as $\boldsymbol{[0,\infty]}$-persistent vector spaces are usually called \define{persistence modules}.
\end{example}
 
 We will see that in  magnitude homology one ``forgets'' the information given by the inclusion maps in the filtration of a simplicial sets, and thus the  chain complexes that one ends up with are more properly \emph{graded} objects, rather than persistent objects.
 
 \begin{definition}
 Let $C$ be a small category, and $I$ a set, which we identify with the discrete category I with objects given by the elements of $I$ and  no morphisms apart from the  identity morphisms. A functor  $ I\to C$ is called an \define{$I$-graded} element of the set of objects of $C$. 
 \end{definition}
 
 If $C$ has all coproducts,  one can characterise such functors as follows:
 
 \begin{proposition}
 Let $C$ be a small category with all coproducts, and let $I$ be a set. There is an isomorphism of categories between the functor category of $I$-graded objects of $C$ and the category with objects pairs $(c,\{c_{i}\}_{i\in I})$ such that $c$ is isomorphic to the coproduct of $\{c_{i}\}_{i\in I}$, and morphisms $(c,\{c_{i}\}_{i\in I})\to (c',\{c'_{i}\}_{i\in I})$ given by $\{f_i\}_{i\in I}$ where for each $i\in I$ we have that $f_i\colon c_i\to c'_i$ is a morphism in $C$. 
 \end{proposition}

\begin{example}\label{Ex:graded chain complex}
When $I=[0,\infty]$, we have that a $[0,\infty]$-graded chain complex of abelian groups can be identified with a chain complex of $[0,\infty]$-graded abelian groups, because coproducts of chain complexes are computed componentwise.
\end{example}

Thus, while a $[0,\infty]$-graded vector space over $\K$ is simply a vector space $V$ together with a direct sum decomposition $V=\oplus_{l\in [0,\infty]}V_l$, we have that a $\boldsymbol{[0,\infty]}$-persistent vector space over $\K$ is a $[0,\infty]$-graded vector space together  with an action of the monoid ring $\K[([0,\infty],+)]$,  which corresponds to the information given by the non-trivial maps $\epsilon\longrightarrow \epsilon'$ whenever $\epsilon\leq \epsilon'$.

\begin{remark}\label{R:graded pers}
Given a monoidal poset $(P,\leq,+)$, and a category $C$ with zero morphisms, we can identify any $P$-graded object in $C$ with a $\boldsymbol{P}$-persistent object in a canonical way.  Namely, consider the full subcategory of the functor category $\Fun(\boldsymbol{P},C)$, given by all functors that send every morphism to the zero morphism in $C$. Then this category is easily seen to be isomorphic to the category of $P$-graded objects of $C$, that is, the functor category $\Fun(P,C)$.
\end{remark}

\section{Coends}\label{S:coends}
One of the main ingredients in the definition of blurred magnitude homology that we will give in Section \ref{S:MH and PH} is the coend, a construction that is ubiquitous in category theory. For ease of reference we briefly recall its definition here.

Intuitively, given a bivariate functor with mixed variance $F\colon D^\op\times D\to C$, its coend is an object in $C$ that identifies the ``left action'' of $F$ with the ``right action'' of $F$; for instance, the tensor product of a left module with a right module over a ring is an example of coend, see  \cite[Section IX.6]{lane1998categories}. 

While one can define a coend in this general setting, we will make use of the following characterisation of coends in the case that $D$ is cocomplete and $C$ small.

\begin{definition}

Suppose that $D$ is a cocomplete category, and $C$ is a small category. Given a functor ${F\colon C^\op\times C\to D}$, its \define{coend} is the coequaliser of the diagram

\[
\begin{tikzpicture}
\node (A) at (0,0) {$\underset{f\colon c\to c'}{\bigsqcup} F(c',c)$};
\node (B) at (3,0) {$\underset{c\in C}{\bigsqcup} F(c,c)\, ,$};

\draw[transform canvas={yshift=1.5ex},->]
(A) -- (B) node[above,midway]{$$};
\draw[transform canvas={yshift=0.5ex},->]
(A) -- (B) node[below,midway]{};
\end{tikzpicture}
\]
\noindent
where the two parallel morphisms are the unique morphisms  induced by the morphisms $F(f,1_{c})\colon F(c',c)\to F(c,c)$, and  ${F(1_{c'},f)\colon F(c',c)\to F(c',c')}$, respectively. If $D$ has additionaly the structure of a monoidal category together with tensor product $\otimes$, then given two functors $L\colon C^{\op}\to D$ and $R\colon C\to D$, we denote the coend of $L\otimes R$ by $L\otimes_C R$. This coend is often referred to as  the \define{functor tensor product} of $L$ and $R$.
\end{definition}

For more details on coends we refer the reader to  \cite[Section IX.6]{lane1998categories}, as well as the  survey \cite{loregian}.


\section{Magnitude homology}\label{S:MH}
Hepworth and Willerton introduced magnitude homology for graphs in \cite{HW17} as the categorification of the magnitude of  a finite metric space associated to a graph. 
Subsequently, Leinster and Shulman generalised magnitude homology to arbitrary finite metric spaces   \cite{LS17}. Here we first briefly recall the definition of magnitude homology as given in \cite{LS17}, and then we give an alternative equivalent definition that 
will serve as the starting point 
 to relate persistent homology to magnitude homology.
 
 \subsection{Magnitude homology for arbitrary finite metric spaces}\label{SS:hochschild}
Instead of with $\boldsymbol{[0,\infty]} $, Leinster and Shulman choose to work with the category $\boldsymbol{[0,\infty)} $  with set of objects given by the non-negative real numbers $[0,\infty)$, with exactly one morphism $\epsilon\longrightarrow \epsilon'$ whenever $\epsilon\leq \epsilon'$, tensor product given by addition, and unit by $0$. See \cite[Section 2]{LS17} for an explanation. Here we will adopt the same choice. In this setting we have that a $\boldsymbol{[0,\infty)}^\op $-category is a quasi-pseudometric space. Leinster and Shulman then give the following definition:

\begin{definition:magA}\cite[Section 3]{LS17}: \label{D:MHA} Let $(X,d)$ be a finite quasi-pseudometric space. 
The \define{magnitude homology} of $X$ is the homology of the chain complex $M(X)$ of $[0,\infty)$-graded abelian groups defined as follows:

\begin{equation}\label{E:bar construction}
M( X)_n =\bigoplus_{l\in [0,\infty)}\Z\left [\left\{(x_0,\dots , x_n) \mid \sum_{i=0}^{n}d(x_i,x_{i+1})=l \right\}\right] \,  .
\end{equation}
Thus,  in degree $n$ it is the free $[0,\infty)$-graded abelian group, which  in degree $l$ is generated by the ordered tuples $(x_0,\dots , x_n)$ of length exactly $l$. Furthermore, the boundary map $d_n\colon M ( X)_n \to M ( X)_{n-1}$ is given by the alternating sum of maps $d_n^i$, defined as follows  for all   $1\leq i \leq n-1$: 
\[
d_n^i((x_0,\dots , x_n))=\begin{cases}
(x_0,\dots ,x_{i-1},x_{i+1},\dots  x_n), & \text{if } d(x_{i-1},x_i)+d(x_i,x_{i+1})=d(x_{i-1},x_{i+1})\\
0, & \text{otherwise}
\end{cases}
\]
while for $i=0$ we have
\begin{alignat}{2}
d_n^0((x_0,\dots , x_n))=\begin{cases}
(x_1,x_2,\dots  x_n), & \text{if } d(x_{0},x_1)=0\\
0, & \notag \text{otherwise}
\end{cases}
\end{alignat}
\noindent
and similarly for $i=n$.

\end{definition:magA}

The assignment $X\mapsto H_\star(M(X))$ induces a functor from the category with objects $\boldsymbol{[0,\infty)}$-categories and morphisms $\boldsymbol{[0,\infty)}$-functors to the category of $[0,\infty)$-graded abelian groups \cite[Theorem 5.12]{LS17}.
Futhermore, for finite quasi-pseudometric spaces $X$, the magnitude homology of $X$ categorifies the magnitude of $X$ with respect to the canonical size function, see
\cite[Theorem 3.5, Corollary 7.15]{LS17}.

\subsection{Magnitude homology: an alternative viewpoint}\label{SS:alternative}

In online discussions \cite{ncafe} Leinster and Shulman initially  gave a different definition of magnitude homology. Here we recall this definition (Definition B), and prove that an adaptation of it (Definition B') agrees with the definition given in the previous section (Definition A). We will  use Definition B' of magnitude homology to relate magnitude homology to persistent homology.

Denote by $\Ab$  the category of abelian groups with monoidal structure given by the  tensor product of abelian groups, which we denote by $\boxtimes$; this induces a monoidal structure on the category of chain complexes over $\Ab$, which we again denote by $\boxtimes$.
Given a $\boldsymbol{[0,\infty)}^\op$-category $X$, Leinster and Shulman  consider the following functor
\begin{equation}\label{E:magn cc}
CN(X)=\left (\boldsymbol{[0,\infty)} \overset{N(X)}{\longrightarrow} \sSet  \overset{\Z[ \cdot ]}{\longrightarrow} \sAb \overset{U}{\longrightarrow} \ch_\Ab \right ) \, ,
\end{equation}
\noindent
where the functor $\Z[ \cdot ]$ is induced  by the free abelian group functor, and the functor $U$ is the functor that sends a simplicial abelian group to its unnormalised chain complex. They then introduce functors of coefficients $A\colon \boldsymbol{[0,\infty)}^\op \longrightarrow \Ab$, 
where one views $A$ as taking values in $\ch_\Ab$ through the canonical inclusion $\Ab\hookrightarrow \ch_\Ab$, and  give the following definition:

\begin{definition:magB}\label{D:MHB}
The \define{magnitude homology of $X$ with coefficients in $A$} is the homology of the chain complex given by the coend $CN(X)\otimes_{\boldsymbol{[0,\infty)}} A$.
\end{definition:magB}

One can describe this chain complex as follows for a particular choice of coefficient functor.

\begin{lemma}\label{P: equivalence MH defn}

For any $\epsilon\in [0,\infty)$ define the following functor of coefficients
\begin{alignat}{2}
A_\epsilon\colon \boldsymbol{[0,\infty)}^\op &\notag\to \Ab  \\
l&\notag \mapsto 
\begin{cases}
\Z,  \,\text{if }l=\epsilon\\
0, \,\text{otherwise}
\end{cases}\\
(\ell\geq \ell') &\notag \mapsto 
\begin{cases}
\id_\Z, \, \text{if } \ell=\ell'=\epsilon\\
0, \,\text{otherwise} \; .
\end{cases}
\end{alignat}
We consider $A_\epsilon$ as taking values in $\ch_\Ab$ through the canonical inclusion functor ${\Ab\hookrightarrow \ch_\Ab}$.

Then, the chain complex $CN(X)\otimes_{\boldsymbol{[0,\infty)}} A_\epsilon$ is given in degree $n$ by the free abelian group on the tuples $(x_0,\dots , x_n)$ that have length exactly $\epsilon$. The boundary maps 
\[
d_n\colon {\left (CN(X)\otimes_{\boldsymbol{[0,\infty)}} A_\epsilon\right )}_n \to \left ({CN(X)\otimes_{\boldsymbol{[0,\infty)}} A_\epsilon}\right )_{n-1}
\] are  alternating sums of maps $d^i_n$ which can be described as follows, for $0<i<n$:
\begin{alignat}{2}
d^i_n( 
(x_0,\dots , x_n))&\notag=  \begin{cases}
(x_0,\dots , x_{i-1},x_{i+1},\dots , x_n), &\text{if } d(x_{i-1},x_{i+1})=d(x_{i-1},x_i)+d(x_i,x_{i+1})\\
0, &\text{otherwise}
\end{cases}
\end{alignat}
\noindent
while for $i=0$ we have
\begin{alignat}{2}
d^0_n\colon {\left (CN(X)\otimes_{\boldsymbol{[0,\infty)}} A_\epsilon\right )}_n&\notag \to {\left (CN(X)\otimes_{\boldsymbol{[0,\infty)}} A_\epsilon\right )}_{n-1}\\
(x_0,\dots , x_n)&\notag \mapsto \begin{cases}
(x_1, x_2,\dots , x_n), & \text{if } d(x_0,x_1)=0\\
0, & \text{otherwise}
\end{cases}
\end{alignat}
and similarly for $i=n$.

\end{lemma}

\begin{proof}\footnote{We note that parts of this proof were given by Shulman in an online comment \cite{ncafe2}.} 
The coend $CN(X)\otimes_{\boldsymbol{[0,\infty)}} A_\epsilon$ is the coequaliser of the following diagram:

\[
\begin{tikzpicture}
\node (A) at (0,0) {$\underset{\ell \leq \ell'}{\bigoplus} CN(X)(\ell)\boxtimes A_\epsilon(\ell')$};
\node (B) at (9,0) {$\underset{\ell\in [0,\infty)}{\bigoplus} CN(X)(\ell)\boxtimes A_\epsilon(\ell)\, .$};
\draw[transform canvas={yshift=1.5ex},->]
(A) -- (B) node[above,midway]{$\underset{\ell \leq \ell'}{\bigoplus}  CN(X)(\ell \leq \ell')\boxtimes 1$};
\draw[transform canvas={yshift=0.5ex},->]
(A) -- (B) node[below,midway]{$\underset{\ell \leq \ell'}{\bigoplus}  1\boxtimes A_\epsilon(\ell \leq \ell')$};
\end{tikzpicture}
\]

Thus, it is the coproduct over $\ell\in [0,\infty)$ of the chain complexes $CN(X)(\ell)\boxtimes A_\epsilon(\ell)$, modulo the relations given by equating the two parallel morphisms on the left hand side. By definition of $A_\epsilon$, the morphisms are both trivial if $\ell'\ne \epsilon$, thus we assume that $\ell'=\epsilon$. The two morphisms are identical if $\ell=\epsilon$, thus we assume that $\ell\ne \epsilon$, and so we have $\epsilon> \ell$. Thus, the bottom parallel morphism is zero, while the top parallel morphism is 
\[
CN(X)(\ell)\longrightarrow CN(X)(\epsilon)\, .
\]
Furthermore, we have
\[ \bigoplus_{\ell\in [0,\infty)} CN(X)(\ell)\boxtimes A_\epsilon(\ell)
 \cong CN(X)(\epsilon)
\]
since tensoring with $A_\epsilon(\ell)$ makes all summands vanish, except if $\ell=\epsilon$.  Thus, in degree $n$ the chain complex $CN(X)\otimes_{\boldsymbol{[0,\infty)}} A_\epsilon$ is the free abelian group on the tuples $(x_0,\dots , x_n)$ that have length exactly $\epsilon$. 

Now, denote by $D(\epsilon)$ the subcomplex of $CN(X)(\epsilon)$ whose $n$-chains are the $n$-tuples with length strictly less than $\epsilon$, so that $CN(X)\otimes_{\boldsymbol{[0,\infty)}} A_\epsilon\cong CN(X)(\epsilon)/D(\epsilon)$ by the previous discussion.
Note that the boundary map on the quotient chain complex $CN(X)(\epsilon)/D(\epsilon)$ is the alternating sum of maps
\[
d^i_n\colon CN(X)(\epsilon)_n/D(\epsilon)_n\to CN(X)(\epsilon)_{n-1}/D(\epsilon)_{n-1}
\]
which send $c+D(\epsilon)_n$ to $d^{i,C}_n(c)+D(\epsilon)_{n-1}$, where $d^{i,C}_n\colon CN(X)(\epsilon)_n\to CN(X)(\epsilon)_{n-1}$ is the map induced by the $i$th face map. Thus $d^i_n(x_0,\dots , x_n)$ is the map induced by the $i$th face maps if and only if by deleting the $i$th entry the length of the tuple is unchanged, and is the zero map otherwise.

\end{proof}

Our aim is to relate Definition B with Definition A. For any $\epsilon\leq \epsilon'$ define the natural transformation $\iota\colon A_{\epsilon}\Rightarrow A_{\epsilon'}$ where $\iota_\ell$ is the identity if $\ell=\epsilon=\epsilon'$, and the zero map otherwise. This induces a chain map 
\[{CN(X)\otimes_{\boldsymbol{[0,\infty)}} A_\epsilon\to CN(X)\otimes_{\boldsymbol{[0,\infty)}} A_{\epsilon'}}\, ,
\] 
and we thus have a functor 
\[
{CN(X)\otimes_{\boldsymbol{[0,\infty)}} A_{-}\colon \boldsymbol{[0,\infty)}\to \ch_\Ab}
\] that assigns to a number $\epsilon$  the chain complex $CN(X)\otimes_{\boldsymbol{[0,\infty)}} A_{\epsilon}$.

Recall that a chain complex   of $[0,\infty)$-graded abelian groups 
can be identified with an $[0,\infty)$-graded chain complex (see Example \ref{Ex:graded chain complex}). 
Thus, in particular, we can identify the chain complex $M(X)$ with a functor $\boldsymbol{[0,\infty)}\to \ch_\Ab$ that coincides with  $M(X)$ on the set of objects, and sends every non-identity morphism to the trivial chain map (see Remark \ref{R:graded pers}). We have:

\begin{proposition}\label{P:MH eq}
The functors $CN(X)\otimes_{\boldsymbol{[0,\infty)}} A_{-}$ and $M(X)$ are isomorphic. In particular, the magnitude homology of $X$ (Definition A) is isomorphic to the homology of $CN(X)\otimes_{\boldsymbol{[0,\infty)}} A_{-}$.
\end{proposition}

\begin{proof}
By Lemma \ref{P: equivalence MH defn} and \eqref{E:bar construction} the chain complexes $M(X)_\epsilon$ and $CN(X)\otimes_{\boldsymbol{[0,\infty)}} A_\epsilon$ are canonically isomorphic. Next, for $\epsilon\leq \epsilon'$, consider the following diagram:
\[
\begin{tikzpicture}
\node (1) at (0,0) {$M(X)_\epsilon$};
\node (2) at (4,0) {$M(X)_{\epsilon'}$};
\node (3) at (0,-2) {$CN(X)\otimes_{\boldsymbol{[0,\infty)}} A_\epsilon$}; 
\node (4) at (4,-2) {$CN(X)\otimes_{\boldsymbol{[0,\infty)}} A_{\epsilon'}$}; 
\draw[->]
(1) edge  (2)
(1) edge  (3)
(3) edge  (4)
(2) edge  (4);
\end{tikzpicture}
\]
\noindent
where the vertical arrows are  the canonical isomorphisms, while the horizontal arrows are zero maps. Thus every such square commutes, so the canonical isomorphisms assemble into a natural isomorphism between $CN(X)\otimes_{\boldsymbol{[0,\infty)}} A_{-}$ and $M(X)$.
\end{proof}
For ease of reference, we state here the equivalent definition of magnitude homology, as given by Proposition \ref{P:MH eq}:

\begin{definition:magBb}\label{D:magBb}
The \define{magnitude homology of $X$} is the homology of the $[0,\infty)$-graded chain complex  $CN(X)\otimes_{\boldsymbol{[0,\infty)}} A_{-}$. 

\end{definition:magBb}

\section{Persistent homology}\label{S:PH}
Persistent homology is, in an appropriate sense, the generalisation of simplicial homology of a simplicial set to persistent simplicial sets. Given a metric space $(X,d)$, we seek to study its geometric and topological  properties by associating to it $\boldsymbol{[0,\infty)}$-persistent simplicial sets $S(X)$.
  
We then consider  the functor
\[
CS(X)=\left (\boldsymbol{[0,\infty)} \overset{S(X)}{\longrightarrow} \sSet  \overset{\Z[ \cdot ]}{\longrightarrow} \sAb \overset{U}{\longrightarrow} \ch_\Ab \right ) \, ,
\]

\noindent
where $\Z[\cdot]$ and $U$ are defined as in \eqref{E:magn cc}. 

Let  $\F$ be a field. 
Consider the constant functor of coefficients
\begin{align}
A\colon \boldsymbol{[0,\infty)} &\notag \to  \ch_\Ab 
\end{align}

\noindent
that sends $\ell$ to the chain complex with a copy of $\F$ concentrated in degree zero, and sends $ \ell \leq\ell'$ to the identity chain map.

The composite $H_\star(CS(X)\boxtimes A)$,  where $H_\star\colon \ch_\Ab\to \Ab$ is the usual homology functor,  is usually called the  ``persistent homology of $X$ (with respect to $S$) with coefficients in $\F$.'' 
Using this coefficient functor has the advantage that, under appropriate finiteness conditions, isomorphism classes of such functors can be completely characterised by a collection of intervals, called the \define{barcode}, see e.g., \cite[Theorem 1.9]{Ou15}. We note that there are many different types of filtered spaces that are used in applications of persistent homology, see \cite[Table 1]{roadmap} for an overview of some of these. To be useful in applications, such spaces have to satisfy   theoretical guarantees dictated by what is called ``topological inference'', see  \cite[Chapter 2 and 5]{Ou15}.

More generally,  we give the following definition:

\begin{definition} Let $(X,d)$ be a metric space,  let $S(X)$ be  a $\boldsymbol{[0,\infty)}$-persistent simplicial set, and $A\colon \boldsymbol{[0,\infty)}\to \ch_\Ab$ a functor. The \define{persistent homology of $X$ with respect to $S$ and with coefficients in $A$} is the composition $H_\star(CS(X)\boxtimes A)$. When $A$ is the unit for $\boxtimes$  we call the homology of $CS(X)\boxtimes A$  the  \define{persistent homology of $X$ (with respect to $S$)}.
\end{definition}

For arbitrary coefficient functors one in general no longer has a barcode. 
However, 
such  functors of coefficients might be interesting for applications, as they might allow to capture more refined information, for instance different torsion or orientability phenomena over different filtration scales, which might be detected  by taking, e.g., coefficients over $\F_2$ (the field with two elements) over a certain interval $I\subset [0,\infty)$, and coefficients over $\F_3$ over a different disjoint interval $J\subset [0,\infty)$. More complicated coefficient functors might allow for an even more refined analysis.

\section{Magnitude meets persistence}\label{S:MH and PH}

In the final section of \cite{LS17} Leinster and Shulman list a series of open problems; two of these problems, as stated in a first version of their manuscript  are as follows (these problems now appear with a change of wording in Section 8 of the published version of the manuscript):
{\it
\noindent
\begin{itemize}  
\item Magnitude homology only ``notices'' whether the triangle inequality is a strict equality or not. Is there a ``blurred'' version that notices ``approximate equalities''?
\item Almost everyone who encounters both magnitude homology and persistent homology feels that there should be some relationship between them. What is it?
\end{itemize}
}
\label{questions}

In this section we attempt a first answer to these questions, which we believe are intertwined: it is the blurred version of magnitude homology that is related to persistent homology. Indeed, as is apparent from Proposition \ref{P:MH eq}, the magnitude homology of a metric space $X$ is a homology theory that in a certain sense forgets the  maps induced on the homology groups by the inclusions of simplicial sets $N(X)(\epsilon)\to NX(\epsilon')$, whenever $\epsilon\leq \epsilon'$, whereas the ``persistence'' in persistent homology is exactly the information given by such maps. Thus, morally, these are very different homology theories.
 
Our starting point  is Definition B' of magnitude, which we adapt to coefficient functors not supported at points, but on intervals. 
\begin{definition}

For any $\epsilon\in [0,\infty]$ define the  functor of coefficients
\begin{alignat}{2}
A_{[0,\epsilon]}\colon 
\boldsymbol{[0,\infty)}^\op  &\notag \to \Ab\\
\ell &\notag \mapsto \begin{cases}
\Z,  \,\text {if } \ell\in [0,\epsilon]\\
0, \,\text{otherwise}
\end{cases}\\
(\ell\geq \ell') &\notag \mapsto 
\begin{cases}
\id_\Z, \, \ell,\ell'\in [0,\epsilon] \\
0, \,\text{otherwise} \; .
\end{cases}
\end{alignat}

\noindent
We  consider $A_{[0,\epsilon]}$ as taking values in $\ch_\Ab$ through the canonical inclusion functor $\Ab\hookrightarrow \ch_\Ab$.
\end{definition}

Now, for any $\epsilon\leq \epsilon'$ we consider the natural transformation $\iota\colon A_{[0,\epsilon]}\Rightarrow A_{[0,\epsilon']}$ where $\iota_\ell$ is the identity if $\ell\in [0,\epsilon]$, and the zero map otherwise. This natural transformation induces a chain map ${CN(X)\otimes_{\boldsymbol{[0,\infty)}} A_{[0,\epsilon]}\to CN(X)\otimes_{\boldsymbol{[0,\infty)}} A_{[0,\epsilon']}}$, 
and we thus have a functor ${CN(X)\otimes_{\boldsymbol{[0,\infty)}} A_{[0,-]}\colon \boldsymbol{[0,\infty)}\to ch_\Ab}$ that assigns to a number $\epsilon$  the chain complex ${CN(X)\otimes_{\boldsymbol{[0,\infty)}} A_{[0,\epsilon]}}$\footnote{We note that, more generally, coends are functorial, see for instance \cite[Remark 1.1.7]{loregian}.}. 
Explicitly, we can describe the chain complexes $CN(X)\otimes_{\boldsymbol{[0,\infty)}} A_{[0,\epsilon]}$ as follows:

\begin{lemma}\label{L:blurred complex}
For any $\epsilon\in [0,\infty)$ we have
\[
{CN(X)\otimes_{\boldsymbol{[0,\infty)}} A_{[0,\epsilon]}}_n=\Z \left [\left \{(x_0,\dots , x_n) \mid \sum_{i=0}^n d(x_i,x_{i+1})\leq \epsilon \right\}\right ]
\]
with boundary maps given by alternating sums of maps  induced by the face maps. 
\end{lemma}
\begin{proof}
The chain complex $CN(X)\otimes_{\boldsymbol{[0,\infty)}} A_{[0,\epsilon]}$ is the coequaliser of the following diagram:

\[
\begin{tikzpicture}
\node (A) at (0,0) {$\underset{\ell\leq \ell'}{\bigoplus} CN(X)(\ell)\boxtimes A_{[0,\epsilon]}(\ell')$};
\node (B) at (9,0) {$\underset{\ell\in [0,\infty)}{\bigoplus} CN(X)(\ell)\boxtimes A_{[0,\epsilon]}(\ell)\, .$};
\draw[transform canvas={yshift=1.5ex},->]
(A) -- (B) node[above,midway]{$\underset{\ell\leq \ell'}{\bigoplus}CN(X)(\ell\leq \ell')\boxtimes 1$};
\draw[transform canvas={yshift=0.5ex},->]
(A) -- (B) node[below,midway]{$\underset{\ell\leq \ell'}{\bigoplus}1\boxtimes A_{[0,\epsilon]}(\ell\leq \ell')$};
\end{tikzpicture}
\]

First, note that  
\[
\bigoplus_{\ell\in [0,\infty)} CN(X)(\ell)\boxtimes A_{[0,\epsilon]}(\ell)\cong \bigoplus_{\ell\in [0,\epsilon]} CN(X)(\ell) \, ,
\]
 as the summands vanish if $\ell> \epsilon$, by definition of $A_{[0,\epsilon]}$. We next see what relations are given by the two parallel morphisms in the diagram. For $\ell'> \epsilon$ we have that the morphisms are both zero, so we assume that $\ell'\in [0,\epsilon]$. Furthermore, the morphisms are identical if $\ell=\ell'$, so we assume that $\ell\ne \ell'$, and thus have $0\leq \ell< \ell'\leq \epsilon$. Thus, the top horizontal morphism is $CN(X)(\ell)\to CN(X)(\ell')$, while the bottom morphism is $CN(X)(\ell)\to CN(X)(\ell)$. By equating these morphisms in $\oplus_{\ell\in [0,\epsilon]} CN(X)(\ell)$ we are thus identifying the  summand $CN(X)(\ell)$ with the image of the inclusion of $CN(X)(\ell)$ in $CN(X)(\ell')$. We thus obtain
 \[
 CN(X)\otimes_{\boldsymbol{[0,\infty)}} A_{[0,\epsilon]}\cong CN(X)(\epsilon) \, .
 \]
Similarly, the relations given by the pair of parallel morphisms tell us that the boundary maps on the quotient chain complex are the boundary maps of the chain complex $CN(X)(\epsilon)$, thus  alternating sums of maps induced by face maps.
\end{proof}

\begin{definition}
Let $(X,d)$ be a metric space.
The \define{blurred magnitude homology} of $X$ is  the homology of   $CN(X)\otimes_{\boldsymbol{[0,\infty)}} A_{[0,-]})$. 
\end{definition}

We  have:

\begin{thm}\label{T:blurred PH}
The functors $CN(X)\otimes_{\boldsymbol{[0,\infty)}} A_{[0,-]}$ and $CN(X)$ are isomorphic.
In particular, the blurred magnitude homology of $X$ is isomorphic to the persistent homology of $X$ with respect to the enriched nerve.
\end{thm}

\begin{proof}
By Lemma \ref{L:blurred complex} we know that there is 
 an isomorphism between the chain complexes ${CN(X)\otimes_{\boldsymbol{[0,\infty)}} A_{[0,\epsilon]}}$  and $CN(X)(\epsilon)$ for any $\epsilon\in [0,\infty)$.
Next,  for any $\epsilon \leq \epsilon'$, consider the square
\[
\begin{tikzpicture}
\node (1) at (0,0) {${CN(X)\otimes_{\boldsymbol{[0,\infty)}} A_{[0,\epsilon]}}$};
\node (2) at (4,0) {$CN(X)\otimes_{\boldsymbol{[0,\infty)}} A_{[0,\epsilon']}$};
\node (3) at (0,-2) {$CN(X)(\epsilon)$}; 
\node (4) at (4,-2) {$CN(X)(\epsilon')$}; 
\draw[->]
(1) edge  (2)
(1) edge  (3)
(3) edge  (4)
(2) edge  (4);
\end{tikzpicture}
\]
where the vertical morphisms are the aforementioned isomorphisms, the top horizontal morphism is given by functoriality of the coend, as discussed before Lemma \ref{L:blurred complex}, and the bottom horizontal morphism is $CN(X)(\epsilon\leq \epsilon')$. The fact that this  square commutes follows by the universal property of coends, see for instance \cite[Definition 1.1.6]{loregian}, applied to the coend in the top left corner.

\end{proof}

\section{Limit homology}\label{S:limit}

In \cite{V26} Vietoris introduced what is now called the Vietoris--Rips complex, as a way to define a homology theory for compact metric spaces\footnote{We note that while Vietoris introduced what is  called the ``Vietoris--Rips simplicial complex'' (at level $\epsilon$) $V(X)(\epsilon)$, here we discuss this homology theory using the  simplicial set $V^\tot(X)(\epsilon)$ associated to it, see  Lemma \ref{L:geom real tot simpl} and Definition \ref{D:VR}.}.
One  starts by considering
the  composition 
 \begin{equation}\label{vr lim cc}
 CV^\tot(X)=\left(\boldsymbol{(0,\infty)} \overset{V^\tot(X)}{\longrightarrow} \sSet \overset{\Z[ \cdot ]}{\longrightarrow} \sAb  \overset{U}{\longrightarrow}  \ch_\Ab \right)\, ,
 \end{equation}
 
\noindent
where $\Z[\cdot]$ and $U$ are defined as in \eqref{E:magn cc}. Here we denote by $\boldsymbol{(0,\infty)}$  the semigroup category associated to the semigroup poset $((0,\infty),\leq, +)$ given by the positive real numbers, with usual addition and order. Vietoris defined the homology of $X$  (for a compact metric space $X$)  to be  the   limit 
 \begin{equation}\label{E:inverse VR}
\mathcal{H}_\star(X)
:=\lim\, H_\star(CV^\simp(X)(\epsilon)) \, .
\end{equation}

  Vietoris's motivation was to prove what is now called the ``Vietoris mapping theorem'', a result that relates the homology groups of two spaces using properties of a map between them. While there has been some work done on Vietoris homology (see, e.g., \cite{H95,reed}), the theory has not  been as widely studied as other homology theories. A limit homology theory that plays a fundamental role in algebraic topology  is  \v{C}ech homology: given a space $X$ and a cover $\UU$ of $X$, one considers the simplicial homology $H_\star(CN(\UU))$ of the nerve of $\UU$. If $\VV$ is a cover of $X$ that refines $\UU$, then there is a homomorphism $H_\star(CN(\VV))\to H_\star(CN(\UU))$. 
 The \v{C}ech homology of $X$ is the  limit over all open covers of $X$.
The difference between Vietoris and \v{C}ech homology is immaterial for compact metric spaces, as for such spaces the homology theories are canonically isomorphic, see \cite{lefschetz}.

In later work, Hausmann \cite{H95} proposed a cohomological counterpart of the homology theory introduced by Vietoris, by considering the  colimit of the functor that one obtains by taking simplicial cohomology of the filtered chain complex in \eqref{vr lim cc}:

\[
\mathcal{H}^\star(X)
:=\colim\,H^\star(CV^\simp(X)(\epsilon))\, .
\] 

Hausmann called this cohomology theory ``metric cohomology'', and not Vietoris cohomology, because the adjective Vietoris had already been used to designate a cohomology theory which is in general not isomorphic to the  cohomological counterpart of the homology theory introduced by Vietoris \cite{H95}. The denomination ``Vietoris--Rips'' for  the  complex introduced by Vietoris  is also due to Hausmann, as the  complex introduced by Vietoris was in the meantime known as Rips complex \cite{R02}.

Instead of the Vietoris--Rips simplicial set, we can consider the enriched nerve associated to a metric space $X$, and take the limit of the resulting homology functor: similarly as in \eqref{vr lim cc}, we consider the composition
 \begin{equation}\label{nerve lim cc}
 CN(X)=\left(\boldsymbol{(0,\infty)} \overset{N(X)}{\longrightarrow} \sSet \overset{\Z[ \cdot ]}{\longrightarrow} \sAb  \overset{U}{\longrightarrow}  \ch_\Ab \right)\, ,
 \end{equation}
 \noindent
and then take the limit of this functor:

\begin{equation}\label{E:inverse nerve}
\lim\, H_\star (CN(X)(\epsilon))\, .
\end{equation}
In the following we relate the  limits \eqref{E:inverse VR} and \eqref{E:inverse nerve}.
\begin{remark}
We note that  since $\boldsymbol{(0,\infty)} $ does not have the structure of a symmetric monoidal category,  the discussion of Remark \ref{R:why monoidal poset} does not apply to the different types of limit homology considered in this section. 
\end{remark}
Let $C$ be any category, and let $(P,\leq ,+)$ be a semigroup poset. Similarly as for monoidal posets, we denote by $\boldsymbol{P}$ the  semigroup category associated to this semigroup poset. The  category with objects given by functors $\boldsymbol{P}\to C$ and morphisms given by natural transformations between them,  can be endowed with an extended pseudo-distance, called \define{interleaving distance} \cite{BSS15}. 
The interleaving distance was first introduced in \cite{C+09} for the monoidal poset $(\mathbb{R},\leq, +)$. 
The central notion is that of interleaving: for $\epsilon\geq 0$ two functors ${M,N\colon \pmb{\mathbb{R}}\to C}$ are \define{$\epsilon$-interleaved} if there are collections of morphisms $\{\phi_\epsilon\colon M(a)\to N(a+\epsilon)\mid a\in \mathbb{R}\}$ and $\{\psi_\epsilon\colon N(a)\to M(a+\epsilon)\mid a\in \mathbb{R}\}$ such that all diagrams of the following form  commute:

\begin{alignat}{3}
&\notag 
\begin{tikzpicture}
\node (1)at (-0.5,0) {$M(a-\epsilon)$};
\node (2) at (1,-1) {$N(a)$};
\node (3) at (2.5,0) {$M(a+\epsilon)$};
\path[->]
(1) edge (2)
(2) edge (3)
(1) edge (3);
\end{tikzpicture}
&&\notag 
\begin{tikzpicture}
\node (1) at (0,0) {$N(a)$};
\node (2) at (3,0) {$N(b)$};
\node (3) at (1,1) {$M(a+\epsilon)$};
\node (4) at (4,1) {$M(b+\epsilon)$};
\path[->]
(1) edge (2)
(1) edge (3)
(2) edge (4)
(3) edge (4);
\end{tikzpicture}
\\
&\notag \begin{tikzpicture}
\node (1)at (-0.5,0) {$N(a-\epsilon)$};
\node (2) at (2.5,0) {$N(a+\epsilon)$};
\node (3) at (1,1) {$M(a)$};
\path[->]
(1) edge (2)
(3) edge (2)
(1) edge (3);
\end{tikzpicture}
&&\notag
\begin{tikzpicture}
\node (1) at (0,0) {$M(a)$};
\node (2) at (3,0) {$M(b)$};
\node (3) at (1,-1) {$N(a+\epsilon)$};
\node (4) at (4,-1) {$N(b+\epsilon) \, .$};
\path[->]
(1) edge (2)
(1) edge (3)
(2) edge (4)
(3) edge (4);
\end{tikzpicture} 
\end{alignat}

\noindent
 Two functors that are $\epsilon$-interleaved have bounded interleaving distance \cite[Theorem 4.4]{C+09}. 

In many examples of filtered spaces that one considers in topological data analysis, what one obtains is not an interleaving of the corresponding homologies, but rather what is called an ``approximation''. 
For $c\geq 1$ two functors $M,N\colon \boldsymbol{[0,\infty)}\to C$ are \define{$c$-approximations} of each other if there are collections of morphisms $\{\phi_c\colon M(a)\to N(ca)\mid a\in \mathbb{R}_{\geq 0}\}$ and $\{\psi_c\colon N(a)\to M(ca)\mid a\in \mathbb{R}_{\geq 0}\}$ such that all diagrams of the following form  commute:

\begin{alignat}{3}
&\notag 
\begin{tikzpicture}
\node (1)at (-0.5,0) {$M(a)$};
\node (2) at (1,-1) {$N(ca)$};
\node (3) at (2.5,0) {$M(c^2a)$};
\path[->]
(1) edge (2)
(2) edge (3)
(1) edge (3);
\end{tikzpicture}
&&\notag 
\begin{tikzpicture}
\node (1) at (0,0) {$N(a)$};
\node (2) at (3,0) {$N(b)$};
\node (3) at (1,1) {$M(ca)$};
\node (4) at (4,1) {$M(cb)$};
\path[->]
(1) edge (2)
(1) edge (3)
(2) edge (4)
(3) edge (4);
\end{tikzpicture}
\\
&\notag \begin{tikzpicture}
\node (1)at (-0.5,0) {$N(a)$};
\node (2) at (2.5,0) {$N(c^2a)$};
\node (3) at (1,1) {$M(ca)$};
\path[->]
(1) edge (2)
(3) edge (2)
(1) edge (3);
\end{tikzpicture}
&&\notag
\begin{tikzpicture}
\node (1) at (0,0) {$M(a)$};
\node (2) at (3,0) {$M(b)$};
\node (3) at (1,-1) {$N(ca)$};
\node (4) at (4,-1) {$N(cb) \, .$};
\path[->]
(1) edge (2)
(1) edge (3)
(2) edge (4)
(3) edge (4);
\end{tikzpicture}
\end{alignat}

It shouldn't then  be too surprising that  functors that  are $c$-approximations of each other have bounded interleaving distance in the log scale \cite{KS13}.

For ease of reference, we state  the definition of $c$-approximations for  functors on the semigroup category $\boldsymbol{(0,\infty)}$:
\begin{definition}
Let $C$ be a category, and  let $M,N\colon \boldsymbol{(0,\infty)}\longrightarrow C$ be two  functors. 
For any $c\geq 1$ denote by $D_c\colon \boldsymbol{(0,\infty)}\to \boldsymbol{(0,\infty)}$ the functor that sends $a$ to $ca$. Furthermore, denote by $\eta_c\colon id_{\boldsymbol{(0,\infty)}} \Rightarrow D_c  $ the natural transformation given by $\eta_c(a)\colon a\to ca $. 
 A \define{$c$-approximation} of $M$ and $N$ is a pair  of natural transformations  
 
 \[
 \phi \colon M \Rightarrow ND_c
 \]
 and
  \[
 \psi \colon N \Rightarrow MD_c
 \]
 such that 
$ (\psi D_c ) \phi = M\eta_{c^2}$ and $(\phi D_c)\psi= N\eta_{c^2}.$
\end{definition}

\begin{lemma}\label{L:interleaving}
Let $M,N\colon \boldsymbol{(0,\infty)}\longrightarrow \Ab$ be two functors. If there exists a $c$-approximation  between $M$ and $N$, then 
\[
\lim\, M(\epsilon) \cong \lim\, N(\epsilon) \, .
\]
\end{lemma}

\begin{proof}

Let
 $\phi \colon M \Rightarrow ND_c$ and  ${\psi \colon N \Rightarrow MD_c}$ be the natural transformations which are part of the data of   the $c$-approximation. These induce  homomorphisms
 \[
\bar{\phi}\colon \lim\, M(\epsilon) \longrightarrow \lim\, N(\epsilon)  
 \]
 
and 
\[
\bar{\psi}\colon  \lim\, N(\epsilon) \longrightarrow \lim\, M(\epsilon) 
\] 
which are inverse to each other.
\end{proof}

\begin{thm}\label{T:nerve vr}
For all  $k=0,1,2,\dots$ there is an isomorphism

 \begin{equation}\label{eq:nerve vr}\notag
 \lim\,  H_{k}(CN(X)(\epsilon))\cong  \mathcal{H}_{k}(X) \, .
 \end{equation}
\end{thm}

\begin{proof}

While there is an inclusion map $N(X)(\epsilon)\to V^\tot(X)(\epsilon)$ for all $\epsilon \in (0,\infty)$, in general there is an inclusion $V^\tot(X)(\epsilon)\to N(X)(c\epsilon)$ only for $c=\infty$, since a $p$-simplex in $V^\tot(X)(\epsilon)$ may have length equal to  $p\epsilon$. On the other hand, when computing simplicial homology in dimension $k$, we only need to consider simplices up to dimension $k+1$, and we will therefore  consider the truncations of the  simplicial sets $N(X)(\epsilon)_{\leq k}$, given by precomposing $N(X)(\epsilon)\colon \Delta^\op\to \Set$ with the inclusion $\Delta_{\leq k}^\op\to \Delta^\op$, and similarly for $V^\tot(X)(\epsilon)$.

The inclusion $N(X)(\epsilon)\to V^\tot(X)(\epsilon)$ induces an inclusion 
\[
\phi_{c}(\epsilon)\colon N(X)_{\leq k+1}(\epsilon)\to V^\tot(X)_{\leq k+1}(c\epsilon)
\]
 for any $c\geq 1$.
If $\sigma$ is a $p$-simplex in $V^\tot(X)(\epsilon)$, then its length is bounded by $p\epsilon$, and thus there is an inclusion map 
\[
\psi_{k+1}(\epsilon)\colon V^\tot(X)_{\leq k+1}(\epsilon)\to N(X)_{\leq k+1}((k+1)\epsilon)\, .
\] 
The collection of  maps 
\[
\left\{\phi_{k+1}(\epsilon)\colon  N(X)_{\leq k+1}(\epsilon)\to V^\tot(X)_{\leq k+1}((k+1)\epsilon) \mid \epsilon \in (0,\infty)\right\}
\]
 and 
 \[
 \left\{\psi_{k+1}(\epsilon)\colon V^\tot(X)_{\leq k+1}(\epsilon)\to N(X)_{\leq k+1}((k+1)\epsilon) \mid \epsilon \in (0,\infty)\right\}
 \]
  are easily seen to satisfy the properties of a $k+1$-approximation, as all maps involved are inclusions. Applying homology we obtain a $k+1$-approximation between the functors $H_k(CV^\tot(X))$ and $H_k(CN(X))$. We can now use Lemma \ref{L:interleaving} 
and obtain an isomorphism
\[
\lim\,  H_k(CN(X)_{\leq k+1}(\epsilon))\cong \lim\,  H_k(CV^\tot(X)_{\leq k+1}(\epsilon)) \, .
\]

Since $H_k(CN(X)_{\leq k+1}(\epsilon))$ is equal to $H_k(CN(X)(\epsilon))$ for all $k$, and similarly for the truncation of the Vietoris--Rips simplicial set, we obtain the claim.

\end{proof}

Finally, we aim to compare blurred and ordinary magnitude homology by using their definition in terms of coends. Thus, similarly as done in the previous part of this section, we consider these as functors on the semigroup category $\boldsymbol{(0,\infty)}$, by precomposing the functors with the inclusion functor ${\boldsymbol{(0,\infty)}\hookrightarrow \boldsymbol{[0,\infty)}}$: we denote by $H_k(CN(X)\otimes_{\boldsymbol{(0,\infty)}} A_{\epsilon})$ and $H_k(CN(X)\otimes_{\boldsymbol{(0,\infty)}} A_{(0,\epsilon]}$ the resulting functors, respectively. 

\begin{corollary}\label{C:MH not PH}
Let $k$ be a non-negative integer, and let $X$ be a metric space with $\mathcal{H}_{k}(X)\ncong 0$. Then 
\[
\lim\, H_k(CN(X)\otimes_{\boldsymbol{(0,\infty)}} A_{\epsilon})\ncong \lim\, H_k(CN(X)\otimes_{\boldsymbol{(0,\infty)}} A_{(0,\epsilon]}) \, .
\]
That is, under the limit,  the $k$th ordinary and blurred magnitude homology of $X$ are not isomorphic. In particular, for any finite metric space the limits differ for $k=0$.
\end{corollary}

\begin{proof}
First, note that 
\[
\lim\, H_k(CN(X)\otimes_{\boldsymbol{(0,\infty)}} A_{\epsilon})\cong 0
\]
since for any $0<\epsilon\leq \epsilon'$ we have that  $CN(X)\otimes_{\boldsymbol{(0,\infty)}} A_{\epsilon}\to CN(X)\otimes_{\boldsymbol{(0,\infty)}} A_{\epsilon'}$ is the zero chain map.  By Lemma \ref{L:blurred complex} we have   
${CN(X)\otimes_{\boldsymbol{(0,\infty)}} A_{(0,\epsilon]}\cong CN(X)(\epsilon)} $   for any $\epsilon>0$, and further by 
 Theorem \ref{T:nerve vr} we have that  
 $
\lim\, H_k(CN(X)(\epsilon))
$
is isomorphic to the Vietoris homology of $X$. 
\end{proof}

\section{Conclusion}

In this manuscript we relate persistent homology to magnitude homology as two different ways of computing homology of  filtered simplicial sets. We give an answer to  two of the open problems formulated by Leinster and Shulman in \cite{LS17}, and listed on Page \pageref{questions} of this manuscript, which we show are intertwined. We define a blurred version of magnitude homology and show that it is the persistent homology taken with respect to a certain  filtered simplicial set. Furthermore, we show how blurred and ordinary magnitude homology differ in the limit: blurred magnitude homology coincides with Vietoris homology, while magnitude homology is trivial.

\section*{Acknowledgments} 

I am profoundly indebted to John Baez, with whom I had originally planned to write this manuscript. 
I am grateful to Jeffrey Giansiracusa, Richard Hepworth, and  Ulrike Tillmann for helpful feedback, as well the anonimous reviewers for their careful feedback and many helpful suggestions. This project was started during my visit to John Baez at the Centre for Quantum Computations (CQT) at NUS in Singapore, and I would like to thank the CQT  for the generous hospitality during my stay. I am grateful to the Emirates Group for sustaining me with an Emirates Award, and to the Alan Turing Institute for support through  an enrichment scholarship,  via 
EPSRC grant EP/N510129/1.

\bibliography{magnitude}

\begin{thebibliography}{10}

\bibitem{BSS15}
P.~Bubenik, V.~de~Silva, and J.~Scott.
\newblock Metrics for generalized persistence modules.
\newblock {\em Foundations of Computational Mathematics}, 15:1501--1531, 2015.

\bibitem{camarena}
Omar~Antol\`in Camarena.
\newblock Unpublished notes available at
  \url{https://www.matem.unam.mx/~omar/notes/ssets-from-complexes.html}
  (accessed: 26 July 2021).

\bibitem{C+09}
F.~Chazal, D.~Cohen-Steiner, M.~Glisse, L.~J. Guibas, and S.~Y. Oudot.
\newblock Proximity of persistence modules and their diagrams.
\newblock pages 237--246. ACM, New York, 2009.

\bibitem{CK17}
R.~Corbet and M.~Kerber.
\newblock The representation theorem of persistence revisited and generalized.
\newblock {\em Journal of Applied and Computational Topology}, 2(1):1--31, Oct
  2018.

\bibitem{C71}
E.~B. Curtis.
\newblock Simplicial homotopy theory.
\newblock {\em Advances in Mathematics}, 6:107--209.

\bibitem{Euler}
L.~Euler.
\newblock {\em Opera Omnia IV.A-4. The Correspondence Euler -- Goldbach.}
\newblock Birkh\"auser {B}asel, 2011.

\bibitem{H95}
J.-C. Hausmann.
\newblock {On the {V}ietoris--{R}ips complexes and a cohomology theory for
  metric spaces}.
\newblock In {\em Prospects in Topology: Proceedings of a Conference in Honor
  of William Browder}, pages 175--188. Princeton U.\ Press, Princeton, 1995.

\bibitem{HW17}
R.~Hepworth and S.~Willerton.
\newblock Categorifying the magnitude of a graph.
\newblock {\em Homology, Homotopy and Applications}, 19(2):31--60, 2017.

\bibitem{kelly}
G.~M. Kelly.
\newblock {\em Basic Concepts of Enriched Category Theory}, volume~64 of {\em
  Lecture Notes in Mathematics 64}.
\newblock Cambridge University Press, 1982.

\bibitem{KS13}
M.~Kerber and R.~Sharathkumar.
\newblock Approximate \v{C}ech complex in low and high dimensions.
\newblock In {\em Algorithms and Computation --- 24th International Symposium,
  ISAAC 2013, Hong Kong, China}, pages 666--676. 2013.

\bibitem{L73}
F.~W. Lawvere.
\newblock {\em Rendiconti del Seminario Matematico e Fisico di Milano},
  {XLIII}:135--166, 1973.

\bibitem{lefschetz}
S.~Lefschetz.
\newblock {\em Algebraic Topology}.
\newblock AMS books online. American Mathematical Society, 1942.

\bibitem{L08}
T.~Leinster.
\newblock The {E}uler characteristic of a category.
\newblock {\em Documenta Mathematica}, 13:21--49, 2008.

\bibitem{L13}
T.~Leinster.
\newblock The magnitude of metric spaces.
\newblock {\em Documenta Mathematica}, 18:857--905, 2013.

\bibitem{LM17}
T.~Leinster and M.~W. Meckes.
\newblock The magnitude of a metric space: from category theory to geometric
  measure theory.
\newblock In {\em Measure Theory in Non-Smooth Spaces}, pages 156--193. 2017.

\bibitem{ncafe}
T.~Leinster and M.~Shulman.
\newblock {Online discussion at the $n$-Category Caf\'{e}}.
\newblock
  \url{https://golem.ph.utexas.edu/category/2016/09/magnitude_homology.html}.

\bibitem{LS17}
T.~{Leinster} and M.~{Shulman}.
\newblock {Magnitude homology of enriched categories and metric spaces}.
\newblock 2017.
\newblock To appear in {\em Algebraic $\&$ Geometric Topology}, preprint
  available at arXiv:1803.05062.

\bibitem{loregian}
F.~{Loregian}.
\newblock {This is the (co)end, my only (co)friend}.
\newblock {\em ArXiv e-prints}, 2015.
\newblock arXiv:1501.02503.

\bibitem{lane1998categories}
S.~Mac~Lane.
\newblock {\em Categories for the Working Mathematician}.
\newblock Graduate Texts in Mathematics. Springer New York, 1998.

\bibitem{M01}
J.~P. May.
\newblock The additivity of traces in triangulated categories.
\newblock {\em Advances in Mathematics}, 163(1):34 -- 73, 2001.

\bibitem{MS05}
E.~Miller and B.~Sturmfels.
\newblock {\em Combinatorial Commutative Algebra}.
\newblock 2005.

\bibitem{roadmap}
N.~Otter, M.~A. Porter, U.~Tillmann, P.~Grindrod, and H.~A. Harrington.
\newblock A roadmap for the computation of persistent homology.
\newblock {\em EPJ Data Science}, 6, 2017.

\bibitem{Ou15}
S.~Y. Oudot.
\newblock {\em {Persistence Theory: From Quiver Representations to Data
  Analysis}}, volume 209 of {\em AMS Mathematical Surveys and Monographs}.
\newblock American Mathematical Society, 2015.

\bibitem{reed}
R.~E. Reed.
\newblock Foundations of {V}ietoris homology with applications to non-compact
  spaces, 1980.
\newblock PhD thesis, Polska Akademia Nauk, Instytut Matematyczny.

\bibitem{R02}
H.~Reitberger.
\newblock Leopold {V}ietoris (1891--2002).
\newblock {\em Notices of the AMS}, pages 1232--1236, 2002.

\bibitem{S90}
S.~H. Schanuel.
\newblock Negative sets have {E}uler characteristic and dimension.
\newblock Lecture notes in mathematics 1488, pages 379--385. 1990.

\bibitem{ncafe2}
M.~Shulman.
\newblock {Online comment at the $n$-Category Caf\'{e}}.
\newblock
  \url{https://golem.ph.utexas.edu/category/2016/08/a_survey_of_magnitude.html#c051059}.

\bibitem{V26}
L.~Vietoris.
\newblock {\"{U}ber den h\"{o}heren Zusammenhang kompakter R\"{a}ume und eine
  Klasse von zusammenhangstreuen Abbildungen}.
\newblock {\em Mathematische Annalen}, 97:454--472, 1927.

\end{thebibliography}
\bibliographystyle{plain}
\end{document}